\documentclass[11pt,letter]{article}

\usepackage[usenames]{color}

\usepackage{amsmath,amssymb,amsthm}

\newcommand{\I}{\mathbb{I}}
\newcommand{\C}{\mathbb{C}}
\newcommand{\SP}{{\mathbb{S}}}
\newcommand{\R}{\mathbb{R}}
\DeclareMathOperator{\G}{{\mathcal{G}}}

\DeclareMathOperator{\F}{{\mathcal{L}}}
\DeclareMathOperator{\bE}{{\mathbb{E}}}
\DeclareMathOperator{\PP}{{\mathbb{P}}}
\newcommand{\calPP}{{\boldsymbol{\mathcal{P}}}}
\newcommand{\calP}{{\mathcal {P}}}
\newcommand{\GG}{{\mathbb G}}

\newtheorem{theorem}{Theorem}[section]
\newtheorem{lemma}{Lemma}[section]
\newtheorem{proposition}{Proposition}[section]
\newtheorem{corollary}{Corollary}[section]
\newtheorem{remark}{Remark}[section]

\title
{A functional equation whose unknown is {${\calP}([0,1])$}
valued
}

\author{Giacomo Aletti\\
\small ADAMSS Center--Dipartimento di Matematica,\\
\small Universit\`a degli Studi di Milano, Milan, Italy\\
\texttt{giacomo.aletti@unimi.it}
\and Caterina May\\
\small Dipartimento SEMEQ,\\
\small Universit\`a del Piemonte Orientale,
Novara, Italy\\
\texttt{caterina.may@eco.unipmn.it}
\and
Piercesare Secchi\\
\small MOX--Dipartimento di Matematica,\\
\small Politecnico di Milano, Milan, Italy\\
\texttt{piercesare.secchi@polimi.it}}

\date{July 11, 2011}

\begin{document}

\maketitle




\begin{abstract}
We study a functional equation whose unknown maps a Euclidean space
into the space of probability distributions on \([0,1].\)
We prove existence and uniqueness of its
solution under suitable regularity and boundary conditions, we show
that it depends continuously on the boundary datum, and we
characterize solutions that are diffuse on \([0,1].\)
A canonical solution is obtained by means of a Randomly Reinforced Urn with
different reinforcement distributions having equal means.
The general solution to the functional equation defines a new parametric collection of distributions on \([0,1]\) generalizing the Beta family.
\end{abstract}

\section{Introduction}
The present work treats a particular functional equation whose unknown maps a Euclidean space into the space \(\mathcal{P}([0,1])\) of probability distributions on $[0,1]$.

Consider two probability distributions $\mu$ and $\nu$ on the interval $[0,\beta]$, with $\beta>0$, and assume that \(\mu\) and \(\nu\) have the same
positive mean. 
Then, for \((x,y)\) ranging over the subspace $\SP=[0,\infty)^2\setminus\{(0,0)\}$ of \(\R^2,\) define the following equation with {\em parameters} $\mu$ and $\nu$:
\begin{equation}\label{eq:FEgeneralTOTcompact}
x\int_{0}^\beta (\G(x,y)-\G(x+k,y))\mu(dk) +
y\int_0^\beta (\G(x,y)-\G(x,y+k))\nu(dk) =0;
\end{equation}
the unknown is the function
\[
\G:\SP\rightarrow {\mathcal P}\mbox{([0,1])}.
\]

Without additional constraints or requirements, equation \eqref{eq:FEgeneralTOTcompact} in its complete generality admits infinitely many solutions.
For instance, any constant function $\G$ satisfies \eqref{eq:FEgeneralTOTcompact}.
We will show that under suitable regularity and boundary conditions, the problem described by \eqref{eq:FEgeneralTOTcompact} is well-posed
in the sense that its solution exists, it is unique and depends continuously on the boundary datum. Moreover, we will also prove that the solution depends continuously on the parameters \(\mu\) and \(\nu\) and we will characterize a class of solutions \(\G\) mapping the interior of \(\SP\) into the subspace of probability distributions diffuse on \([0,1].\)


 A particular instance of the problem considered in this paper has been studied in \cite{AlettiMaySecchi07} where it is proved that, when the two parameters $\mu$ and $\nu$ are equal, there exists one and only one continuous solution to \eqref{eq:FEgeneralTOTcompact} that maps the \(x\)-axis and \(y\)-axis borders of $\SP$ in the point mass at 1 or at 0, respectively, and that approaches the point mass at $x/(x+y)$ as $x+y$ tends to infinity. We here extend this result to the case of different parameters \(\mu\) and \(\nu\) with the same mean, and to more general boundary conditions. These will be described by means of a continuous function $\varphi: [0,1] \rightarrow \mathcal{P}([0,1])$ that will represent
the {\em boundary datum} of the problem.


From a probabilistic point of view, \eqref{eq:FEgeneralTOTcompact} is naturally connected to the dynamics
of a two-color randomly reinforced urn  with reinforcement
distributions $\mu$ and $\nu.$

Indeed, for \((x,y) \in \SP,\) consider an urn containing initially $x$ black balls and $y$ white balls.
The urn is sequentially sampled. At time $n=1,2,\dots$ a ball is
drawn from the urn and its color is observed: if the sampled ball
is black, it is replaced in the urn together with a random number
of black balls having distribution \(\mu,\) if the sampled ball is white it is
replaced in the urn together with a random number of white balls
having distribution \(\nu.\) The extra balls added every time the urn is sampled are called reinforcements.
This urn scheme is called a two-color randomly reinforced urn (\emph{RRU}); it has been introduced in \cite{DurhamYu,LiDuhramFlournoy,DurhamFlournoyLi,May.et.al.05,Muliere.et.al.06} and
further studied in \cite{Beggs05,AlettiMaySecchi07,AlettiMaySecchi08,MayFl07,PS07} as a general model for learning through reinforcement with direct applications in statistics as a device for adaptive sampling.

If $\{\G(x, y),(x,y)\in\SP\}$ is a family of distributions on \([0,1]\) parameterized by the elements of \(\SP\)  and \(\gamma\) is a distribution on \(\SP,\) we use
\(\G(X,Y) \wedge \gamma\) to indicate the distribution on \([0,1]\) obtained by mixing the distributions \(\G\) according to \(\gamma.\)   This can be thought of as the distribution of a random value in \([0,1]\) generated through a two-step procedure: first sample \((X,Y)\in \SP\) with distribution \(\gamma\) and then, given \((X,Y)=(x,y),\) sample a random value in \([0,1]\) according to the distribution \(\G(x,y).\)

Now let \((X_1,Y_1)\) be the random number of black and white balls respectively present in a \emph{RRU} after it has been sampled for the first time and indicate with \(\gamma_1(x,y)\) the distribution of \((X_1,Y_1)\) on \(\SP,\) which depends on \(\mu, \nu\) and the initial urn composition \((x,y).\)
Finally let \(A\) to be an operator acting on the distributions \(\G\) and defined as
\[
(A\G)(x,y)=\G(X,Y) \wedge \gamma_1(x,y).
\]

A solution \(\G\) of equation \eqref{eq:FEgeneralTOTcompact} is a fixed point of $A: A\G = \G$.
Indeed we will also show that if \(\G(x,y)\) is the distribution of the limit proportion of black balls of a \emph{RRU} with initial composition \((x,y)\) and reinforcement distributions \(\mu\) and \(\nu,\)  then \(\G\) is a fixed point of \(A\)  satisfying specific boundary conditions.

As a prototypical example, consider an \emph{RRU} whose reinforcement distributions \(\mu\) and \(\nu\) are both point masses at 1; this is a P\'olya urn scheme. It is well known that if \((x,y) \in \SP\) is the initial composition of a P\'olya urn, then the limit proportion of black balls generated by this urn scheme has distribution Beta\((x,y).\) In fact the family of Beta distributions is a fixed point of the operator \(A\) related to the P\'olya urn. In this sense, the general solution of \eqref{eq:FEgeneralTOTcompact} defines a collection of distributions on \([0,1]\) parameterized by elements of \(\SP\) and generalizing the Beta.


In the next section we set notation and terminology, we formally describe the functional equation problem and we state three results concerning its solution; they will be proved in the rest of the paper. Section \ref{sect:RRU} deals with the construction of the {\em canonical solution} to \eqref{eq:FEgeneralTOTcompact} for the special case when the boundary datum \(\varphi(t)\) is the point mass at \(t,\) for all \(t \in [0,1];\) indeed, the canonical solution is obtained by means of a \emph{RRU}.
Canonical solutions are the building blocks for proving existence, uniqueness and regularity properties of the solution to the functional equation problem with a general boundary datum; this will be shown in Section \ref{sect:fixed}. Section \ref{sect:diffuse} describes functional equation problems whose solution maps the interior of \(\SP\) into the subspace of \(\calP([0,1])\) consisting of probability distributions with no point masses. The final Section \ref{section:examples} illustrates some examples extending the solution of equation \eqref{eq:FEgeneralTOTcompact} well beyond the case of the Beta family. Auxiliary technical results have been postponed to the Appendix.

\section{Problem and main results}
In this section we set notation and terminology and we describe the functional equation problem in detail. We also state three main results concerning its solution; they will be proved in the rest of the paper.

\subsection{The Wasserstein metric for spaces of probability distributions}
For any $\beta \in (0,\infty)$, we endow the set $\calP([0,\beta])$ of probability distributions on the real interval $[0,\beta]$ with
the $1$--Wasserstein metric $d_W$ which metrizes weak
convergence. Recall that, for \(\xi_1,\xi_2 \in \calP([0,\beta]),\)
\begin{equation*}
d_W(\xi_1,\xi_2)=
\int_0^\beta|F_{\xi_1}(t)-F_{\xi_2}(t)|dt
=\int_0^1|q_{\xi_1}(t)-q_{\xi_2}(t)|dt,
\end{equation*}
where $F_\xi$ 
 and $q_\xi$ are the cumulative distribution function
 and the quantile function of $\xi \in \calP([0,\beta])$, respectively (see \cite{Gibbs.et.al.02} for more details).
Moreover, by the Kantorovich-Rubinstein Theorem,
\begin{equation}\label{eq:KR}
d_W(\xi_1,\xi_2)=
\inf\{E(|X_1-X_2|):X_1 \sim \xi_1,X_2 \sim \xi_2\}
\end{equation}
where the infimum is taken over all joint distributions for the vector of random variables $(X_1,X_2)$
with marginal distributions equal to $\xi_1$ and $\xi_2,$ respectively.
The metric space $(\calP([0,\beta]),d_W)$ is complete and compact.

\subsection{The set $\calPP$ of parameters}\label{sec:parameters}
For \(0< m_0\leq \beta<\infty,\) endow the cartesian product $\calP([0,\beta])\times\calP([0,\beta])$ with the Manhattan-distance
\[
d_{M}((\mu_1,\nu_1),(\mu_2,\nu_2)) = d_W(\mu_1,\mu_2)+d_W(\nu_1,\nu_2)
\]
and consider the subset  $\calPP$ of couples  ($\mu,\nu$) of probability distributions with support in \([0,\beta]\) having means that are equal and that are both greater than or equal to \(m_0\), i.e.  such that
\begin{equation*}
\int_0^\beta k \mu(dk) = \int_0^\beta k \nu(dk) \ge m_0.
\end{equation*}
The elements of \(\calPP\) will act as the parameters for the functional equation \eqref{eq:FEgeneralTOTcompact}; note that
\(\calPP\) is a closed subset of the metric space $\calP([0,\beta])\times\calP([0,\beta])$ and therefore it is compact.

\subsection{The set $C([0,1],\calP([0,1]))$ of boundary data}
A boundary datum $\varphi$ is defined as a continuous map from $[0,1]$ to $\calP([0,1])$.
We endow the set of boundary data $C([0,1],\calP([0,1]))$ with the $\sup$-distance
\[
d_\infty(\varphi_1,\varphi_2) = \sup_{t\in[0,1]} d_W(\varphi_1(t),\varphi_2(t));
\]
then $(C([0,1],\calP([0,1])),d_\infty)$ is a complete metric space.

From now on, $\delta$ will indicate the element of $C([0,1],\calP([0,1]))$ defined by setting \(\delta(t)=\delta_t\) for \(t \in [0,1],\) where \(\delta_t\) denotes the point mass at \(t.\)

\subsection{The set $C(\SP,\calP([0,1]))$ where solutions are to be found}
Let $\SP=[0, \infty)^2\setminus\{(0,0)\}$ and $C(\SP,\calP ([0,1]))$ be the set of the continuous maps
$\G:\SP\rightarrow {\mathcal P}\mbox{([0,1])}$.

For \(n=1,2,\ldots\) let \(\SP_n=\{(x,y)\in \SP: x+y\geq 1/n\}\) and consider the distance between elements
\(\G_1,\G_2 \in C(\SP_n,\calP ([0,1]))\) defined by
\[
d_n(\G_1,\G_2)= \sup_{(x,y) \in \SP_n
} d_W(\G_1(x,y),\G_2(x,y)).
\]
We then define a new distance \(d\) by setting, for all
\(\G_1,\G_2 \in C(\SP,\calP ([0,1])),\)
\[
 d(\G_1,\G_2)=\sum_{n=1}^\infty \frac{1}{2^{n}} \cdot \frac{d_n(\G_1|_{\SP_n},\G_2|_{\SP_n})}{1+d_n(\G_1|_{\SP_n},\G_2|_{\SP_n})}
\]
where \(\G|_{\SP_n}\) indicates the restriction to \(\SP_n\) of a \(\G \in C(\SP,\calP ([0,1])).\)

The distance \(d\) metrizes the uniform weak convergence in any closed subset of $\SP\bigcup \{(0,0)\}$ which does not
contain the origin. Note that convergence with respect to \(d\) is equivalent to convergence with respect to all \(d_n\) of the corresponding restrictions. The set $(C(\SP,\calP ([0,1])),d)$ is a complete metric space; we will look for elements of this space that are solutions of the functional equation \eqref{eq:FEgeneralTOTcompact}.

\subsection{The functional equation problem}

The Problem object of this paper is now easily stated:\\
{\em given} \[ (\mu,\nu) \in \calPP \;\mbox{and the boundary datum}\; \varphi \in C([0,1],\calP([0,1])),\]
{\em find}
\begin{subequations}\label{eq:system1:F}
\begin{equation}
\label{eq:system1:contin}
\G \in C(\SP,\calP ([0,1]))
\end{equation}
such that, for all \((x,y) \in \SP\),
\begin{align}
\label{eq:system1:func_eq}
&\begin{aligned}
x\int_0^\beta& (\G(x,y)-\G(x+k,y))\mu(dk)\\&+y\int_0^\beta (\G(x,y)-\G(x,y+k))\nu(dk) =0,
 \end{aligned}
\\
\label{eq:system1:ax1}
&\G(0,y) = \varphi(0),\\
\label{eq:system1:ax2}
&\G(x,0) = \varphi(1),\\
\label{eq:system1:ax_inf}
&d_W\Big(\G(x,y),\varphi\big(\frac{x}{x+y}\big)\Big)
\mathop{\longrightarrow}_{x+y\to\infty} 0.
\end{align}
\end{subequations}

\subsection{Main results}

Our first result states that Problem \eqref{eq:system1:F} is well-posed in the sense of Hadamard.

\begin{theorem}\label{theo:FEgeneral}
A solution to Problem \eqref{eq:system1:F} exists, it is unique, and it depends continuously on the boundary datum.
\end{theorem}

In the rest of the paper,
we denote with $\G^\varphi_{(\mu,\nu)}$ the unique solution to Problem~\eqref{eq:system1:F}.
Theorem \ref{theo:FEgeneral} will be proved first in the special case when the boundary datum is the map \(\delta.\) Indeed
$\G^\delta_{(\mu,\nu)}$ is a {\em canonical solution} for the problem since,
for any other boundary datum $\varphi\in C([0,1],\calP([0,1]))$,
we will show that
\begin{equation}\label{eq:soluz_gen}
\G^\varphi_{(\mu,\nu)} =\Psi_\varphi(\G^\delta_{(\mu,\nu)}),
\end{equation}
where $$\Psi_\varphi:C(\SP,\mathcal{P}([0,1]))\to
C(\SP,\mathcal{P}([0,1]))$$ is the linear map defined by setting, for all \({\mathcal G} \in  C(\SP,\mathcal{P}([0,1]))\),
\begin{equation}\label{eq:soluz_gen:2}
\Psi_\varphi({\mathcal G})(x,y)= \int_0^1 \varphi(t) {\mathcal G}(x,y) (dt)
\end{equation}
with \((x,y)\) ranging over \(\SP.\)

The second theorem concerns the continuity of the solution to Problem \eqref{eq:system1:F} when the parameters of the equation are let to change.
Indicate with $\GG^\varphi$ the set of solutions to Problem
\eqref{eq:system1:F} obtained by holding fixed the boundary datum $\varphi\in C([0,1],\calP([0,1]))$ and by letting the parameters \((\mu,\nu)\) range over $\calPP$:
\[
 \GG^\varphi = \Big\{ \G^\varphi_{(\mu,\nu)} \colon
(\mu,\nu) \in \calPP \Big\}.
\]
\begin{theorem}\label{theo:contparameters}
For any given boundary datum $\varphi\in C([0,1],\calP([0,1]))$,
the map
$$(\mu,\nu)\mapsto \G^\varphi_{(\mu,\nu)},$$
from $\calPP$ to $\GG^\varphi$, is uniformly continuous and $\GG^\varphi$ is compact.
\end{theorem}

To prove Theorem \ref{theo:contparameters} we will first show that it holds for canonical solutions, i.e. for $\GG^\delta,$ and then we will prove that the map $\Psi_\varphi$
is continuous.

The third result regards a different regularity property of the solution to Problem \eqref{eq:system1:F}, which depends on the boundary datum \(\varphi\) but not on the parameters \((\mu,\nu).\) Indeed we characterize solutions \(\G^\varphi_{(\mu,\nu)}\)
mapping the interior of \(\SP\) into the class of probability distributions on \([0,1]\) having no point masses; such solutions will be called {\em diffuse.}

A boundary datum \(\varphi \in C([0,1],\calP([0,1]))\) is said to be {\em monotonic} if, for all \(s,t \in [0,1],\) \(s\leq t,\)  \[\varphi(s)\leq_{st}\varphi(t).\]
For a given \(\varphi \in C([0,1],\calP([0,1])),\) indicate with \(\Phi\) the probability distribution on \([0,1]\) obtained as
the convex combination with uniform weights of the members of the family \(\{\varphi(t): t \in [0,1]\}\); i.e. \(\Phi=\int_0^1 \varphi(t) dt.\)

\begin{theorem}\label{theo:nomassgeneral}
Assume that the boundary datum \(\varphi\) is monotonic and let \(\G^\varphi_{(\mu,\nu)}\) be the unique solution to Problem \eqref{eq:system1:F}.
Then:
\begin{enumerate}
\item
If there is \((x_0,y_0)\) in the interior of \(\SP\) such that \(\G^\varphi_{(\mu,\nu)}(x_0,y_0)\)
has no point masses in \([0,1],\)
then \(\Phi=\int_0^1 \varphi(t) dt\) has no point masses in \([0,1].\)
\item
If \(\Phi=\int_0^1 \varphi(t) dt\) has no point masses in \([0,1],\) then
\(\G^\varphi_{(\mu,\nu)}(x,y)\) has no point masses in \([0,1]\) for all \((x,y)\) in the interior of \(\SP.\)
\end{enumerate}
\end{theorem}
Once again, in Section \ref{sect:diffuse}, we will first prove Theorem  \ref{theo:nomassgeneral} for canonical solutions and then for the general solution \(\G^\varphi_{(\mu,\nu)}.\)

\section{Existence of canonical solutions: a Randomly Reinforced Urn}
\label{sect:RRU}

In this section we assume that the boundary datum of Problem \eqref{eq:system1:F} is the map \(\delta \in C([0,1],\calP([0,1]));\)
we prove the existence of a solution for this special instance of the problem, by constructing it through a \emph{RRU} scheme.
This solution will be called {\em canonical} since the solution to Problem \eqref{eq:system1:F} for a general boundary datum  will be obtained by transforming the canonical solution through a suitable map.
While constructing canonical solutions, we will also provide two novel results concerning the continuity of the distribution of the limit proportion of black balls generated by an \emph{RRU} by proving its continuity with respect to the initial urn composition (Lemma~\ref{prop:modulo_Ginterno}) as well as with respect to the reinforcement distributions (Proposition~\ref{prop:PtoG}); the continuity is uniform in any closed subset of \(\SP\) which does not contains the origin.

On a rich enough probability space \((\Omega, \mathcal A, P),\) define two independent infinite sequences of random elements,
\(\{U_n\}\) and \(\{(V_n,W_n)\}\); \(\{U_n\}\) is a sequence of i.i.d. random variables uniformly distributed on \([0,1],\) while
\(\{(V_n,W_n)\}\) is a sequence of i.i.d. bivariate random vectors with components uniformly distributed on \([0,1].\)
Then, define an infinite sequence \(\{(R_X(n),R_Y(n))\}\) of bivariate random vectors by setting,
for all \(n,\)
\[
R_X(n)=q_\mu(V_n)\;\;\mbox{and}\;\; R_Y(n)=q_\nu(W_n),
\]
where $q_\mu$ and $q_\nu$ are the quantile functions of two distributions $\mu$ and $\nu$ having support in $[0,\beta],$ with \(\beta >0.\)
Let \(x\) and \(y\) be two non-negative real numbers such that
\(x+y>0.\)
Set  $X_0=x$, $Y_0=y$, and, for $n= 0,1,2,...$, let
\begin{equation}\label{def1}
\left \{ \begin{array}{lll}
  X_{n+1}&=&X_n+R_{X}(n+1)\I(n+1),\\
  Y_{n+1}&=&Y_n+R_{Y}(n+1)(1-\I(n+1)),
\end{array} \right.
\end{equation}
where the variable \(\I(n+1)\) is the indicator
of the event \(\{U_{n+1}\leq X_{n}(X_{n}+Y_{n})^{-1}\}.\)
The law of \(\{(X_n,Y_n)\}\) is that of
the stochastic process counting, along the sampling sequence, the number of black and
white balls present in
a \emph{RRU} with initial
composition \((x,y)\) and reinforcement distributions equal
to \(\mu\) and $\nu,$ respectively.

For \(n=0,1,2,\ldots\) let $D_n=X_n+Y_n$ be the total number of balls present in the urn at time $n$ and set
$$Z_n(x,y)=X_n/D_n;$$ \(Z_n(x,y)\) represents the proportion
of black balls in a \emph{RRU} with initial composition \((x,y),\) before the \((n+1)\)-th ball is sampled from it.
In \cite{Muliere.et.al.06} it is proved that \(\{Z_n(x,y)\}\) is eventually a bounded sub- or super-martingale, and it thus converges almost surely, and in $L^p$, for $1\leq p\leq\infty,$ to a random variable $Z_\infty(x,y)\in [0,1];$ moreover, when $\mu$ and \(\nu\) have different means, $Z_\infty(x,y)$ is the point mass concentrated in $1$ or $0$, according to whether the mean of \(\mu\) is greater or smaller than that of \(\nu.\)
However, when the means of $\mu$ and $\nu$ are the same, the distribution of $Z_\infty(x,y)$ is unknown, apart from a few special cases, see \cite{AlettiMaySecchi07} and \cite{MayFl07}.

For a given couple \((\mu,\nu) \in \calPP,\) let
$$\F_{(\mu,\nu)}:\SP\to\calP([0,1])$$ be the map which assigns to every
$(x,y) \in \SP$ the distribution of the limit proportion $Z_\infty(x,y)$ of a \emph{RRU} with initial composition \((x,y)\) and reinforcement distributions $\mu$ and $\nu$. 
In the special case where $\mu=\nu$, the map $\F_{(\mu,\mu)}$ has been characterized in \cite{AlettiMaySecchi07} as the unique solution to
Problem~\eqref{eq:system1:F} when the boundary datum is \(\delta.\) We now extend this
result to the general case \((\mu,\nu) \in \calPP.\)

\begin{proposition}\label{prop:soluz RRU}
$\F_{(\mu,\nu)}$ is a solution to Problem~\eqref{eq:system1:F} when its boundary datum is equal to $\delta$.
\end{proposition}

In order to prove Proposition \ref{prop:soluz RRU} we need some auxiliary results;
when they do not depend on the parameters $(\mu,\nu)\in\calPP,$ and there is no place for misunderstanding,
we write $\F$ for $\F_{(\mu,\nu)}$. Some technicalities connected with the Doob's decomposition of the process $\{Z_n\}$ have been postponed to the Appendix.

The distance between \(\F,\) evaluated at $(x,y),$ and the boundary datum, evaluated at \(x/(x+y),\) is controlled
in the following lemma; this distance is uniformly bounded, provided that the size of the urn initial composition
is sufficiently large.

\begin{lemma}\label{F all'infinito}
If $x+y\geq 2\beta$,
\begin{equation*}
 d_W\big(\F(x,y),\ \delta_{\frac x {x+y}}\big)< 2\sqrt{\frac{\beta}{x+y}}.
\end{equation*}
\end{lemma}
\begin{proof}
Note that, by \eqref{eq:KR},
$d_W\big(\F(x,y), \delta_{\frac x {x+y}}\big)=\bE(|Z_\infty(x,y)-\frac{x}{x+y}|)$.
Moreover, $\bE(|Z_\infty(x,y)-\frac{x}{x+y}|)=\bE(|A_\infty(x,y)+M_\infty(x,y)|)$,
where $(A_n)_n$ and $(M_n)_n$ are the Doob's decomposition processes
(see Appendix) of $(Z_n)_n$. Note that $\bE(|M_\infty|)\leq
\sqrt{\bE(\langle M\rangle_{\infty})}$
by Jensen inequality and thus,
from triangular inequality,
Lemma~\ref{sub_Martingale}, Lemma~\ref{bracket} with $n=0$,
and
because $x+y\geq 2\beta$
implies $\sqrt{\frac{\beta}{x+y}}<1$,
we get
\begin{equation*}
d_W \big(\F(x,y),\ \delta_{\frac x {x+y}}\big) \leq
\frac{\beta}{x+y} +\sqrt{\frac{\beta}{x+y}} <
\sqrt{\frac{\beta}{x+y}} + \sqrt{\frac{\beta}{x+y}} .
\end{equation*}
\end{proof}

The Markov inequality together with Lemma~\ref{F all'infinito} imply the following
corollary.
\begin{corollary}\label{cor:scostamento}
If $x+y\geq  2\beta,$
\begin{equation*}
\PP\Big(\big|Z_\infty(x,y)-\frac{x}{x+y}\big|>h_0\Big)\leq
\frac{2}{h_0}\sqrt{\frac{\beta}{x+y}}
\end{equation*}
for every \(h_0>0.\)
\end{corollary}
\begin{lemma}\label{lem:Z_l,F}
For all $n_0\ge 1$ and $\epsilon>0$, there is $N=N(\epsilon,n_0)$
such that,
\[
\bE\big(|Z_n(x,y)-Z_\infty(x,y)|\big)\leq \epsilon,
\]
if $n\geq N$ and $x+y\geq 1/n_0.$
\end{lemma}
\begin{proof}
Equation \eqref{eq:1suDN} yields, for all $t>0$,
\begin{equation}\label{eq:DN>t}
\PP (D_n<t) = \PP\Big(\frac{1}{D_n}
>\frac{1}{t}\Big)\leq t
\bE \frac{1}{D_n} < t\frac{1+n_0(\beta-m_0)}{m_0(n-1)+\beta}
\end{equation}
where $m_0$ is given in Section~\ref{sec:parameters}.
Set
\begin{equation}\label{eq:t1}
 t=\max\{16\beta/\epsilon^2,2\beta\}
\end{equation}
and
\begin{equation}\label{eq:N1}
N
\geq \frac{\frac{2t}{\epsilon} (1+n_0(\beta-m_0) )-\beta}{m_0}+1.
\end{equation}
From \eqref{eq:DN>t} and \eqref{eq:N1}, we get
\begin{equation}\label{eq:P(DN>t)}
 \PP(D_N<t)<\frac{\epsilon}{2}.
\end{equation}
Moreover, since the process \(\{(X_n,Y_n)\}\) is Markov, it follows from Lemma~\ref{F all'infinito} and \eqref{eq:t1} that, for
$n\geq N$ and $\omega\in\{D_N \geq t\}$,
\begin{equation}\label{dW(F(XN),delta)}
\bE\Big(|Z_\infty-Z_n|\Big|(X_n,Y_n)\Big)(\omega) =
d_W\Big(\F(X_n(\omega),Y_n(\omega)),\
\delta_{\frac {X_n(\omega)} {X_n(\omega)+Y_n(\omega)}}\Big)
\leq \frac{\epsilon}{2}.
\end{equation}
Since $\{D_{n+1}< t\}\subseteq \{D_{n}< t\}$ for all \(n,\)
\eqref{eq:P(DN>t)} and \eqref{dW(F(XN),delta)} imply that
\begin{multline*}
\bE[|Z_\infty-Z_n|]= \bE[|Z_\infty-Z_n|;\{D_n\geq t\}]+\bE[|Z_\infty-Z_n|;\{D_n<t\}]\\
\begin{aligned}
& \leq \bE\big(\bE(|Z_\infty-Z_n|1_{\{D_n \geq t\}}\big|(X_n,Y_n))\big)+\PP(D_N<t)\\
& \leq \epsilon
\end{aligned}
\end{multline*}
for \(n\geq N.\)
\end{proof}

The next result can be read as
a bound on the modulus of continuity of \(\F\) when evaluated at the inner points of \(\SP.\)

\begin{lemma}\label{prop:modulo_Ginterno}
For all $n_0\ge 1$ and $\epsilon>0$, there is
$
\eta=\eta(\epsilon,n_0),
$
increasing with $\epsilon$ and $1/n_0,$ such that
$$ d_W(\F(x,y),\F(\bar x,\bar y)) < \epsilon,$$
  if $|x-\bar x|+|y-\bar y|<\eta$
 and $\min\{x+y,\bar x+\bar y\}\geq 1/n_0.$
\end{lemma}

\begin{proof}
Let $N=N(\epsilon/4,n_0)$ be given by
Lemma~\ref{lem:Z_l,F}.
Then:
\begin{eqnarray*}
\lefteqn{
d_W(\F(x,y),\F(\bar x,\bar y))}\\
&& 
\leq
E[|Z_\infty(x,y)-Z_N(x,y)|]+ E[|Z_\infty(\bar x,\bar y)-Z_N(\bar x,\bar y)|]
\\
&& 
\qquad\qquad +E[|Z_N( x,y)-Z_N(\bar x,\bar y)|] \\
&& 
\leq \frac{\epsilon}{2} + E[|Z_N( x,y)-Z_N(\bar x,\bar y)|].
\end{eqnarray*}
For controlling the last term,
we adopt a coupling argument as in \cite{AlettiMaySecchi07}.
Consider two different randomly reinforced urns, the first one with initial composition $(x,y)$ and second one with $(\bar x,\bar y)$.
The two urns are coupled in the sense that the same
processes \(\{U_n\}\) and \(\{(V_n,W_n)\}\)
generate both \(\{(X_n(x,y),Y_n(x,y))\}\) and \(\{(X_n(\bar x,\bar
y),Y_n(\bar x,\bar y))\) according to the dynamics described in
\eqref{def1}. 
With the same arguments as in
\cite[pages 701-702]{AlettiMaySecchi07}, one may show  that
$$\bE[|Z_N(x,y)-Z_N(\bar x, \bar y)|]\leq (1+N)\dfrac{|x-\bar x|+|y-\bar y|}{\min\{x+y,\bar x + \bar y\}};$$
therefore, if $\eta \leq \frac{\epsilon}{2(1+N)n_0}$,
$$\bE[|Z_N(x,y)-Z_N(\bar x, \bar y)|]\leq \frac{\epsilon}{2}.$$
\end{proof}

\begin{proof}[Proof of Proposition \ref{prop:soluz RRU}]
By considering the conditional distribution of $Z_\infty(x,y),$ given $\I(1), R_X(1)$ and $R_Y(1)$, and taking the expected values,
one immediately verifies that $\F_{(\mu,\nu)}$ satisfies
equation \eqref{eq:system1:func_eq} for all \((x,y) \in \SP.\)
Conditions \eqref{eq:system1:ax1} and \eqref{eq:system1:ax2} are also easily verified when $\varphi=\delta.$
Finally, \eqref{eq:system1:contin} and \eqref{eq:system1:ax_inf}
are consequences of Lemma~\ref{prop:modulo_Ginterno}
and of Lemma~\ref{F all'infinito}, respectively.
\end{proof}

The next result proves a further regularity property of $\F_{(\mu,\nu)}.$
\begin{proposition}\label{prop:PtoG}
The map $$(\mu,\nu)\mapsto \F_{(\mu,\nu)},$$
from $(\calPP, d_M)$ to $(C(\SP,\calP([0,1])),d)$, is uniformly continuous.
\end{proposition}

\begin{proof}
Let $A:C(\SP,\calP([0,1]))\times \calPP\to C(\SP,\calP([0,1]))$
be the operator defined by setting, for every \(\mathcal H \in C(\SP,\calP([0,1]))\) and \( (\mu,\nu) \in \calPP,\)
\begin{gather*} 
A({\mathcal H},(\mu,\nu))(x,y) =
\tfrac{x}{x+y}\int_0^{\beta}{\mathcal H}(x+k,y)\mu(dk) +
\tfrac{y}{x+y}\int_0^{\beta}{\mathcal H}(x,y+k)\nu(dk)
\\ 
= \tfrac{x}{x+y}\int_0^1{\mathcal H}(x+q_\mu(t),y)dt +
\tfrac{y}{x+y}\int_0^1{\mathcal H}(x,y+q_\nu(t))dt,
\end{gather*}
where \((x,y)\) ranges over \(\SP.\)

Let $n\ge 1$. Then
\begin{equation}\label{eq:opA1r}
d_n(A({\mathcal H}_1,(\mu,\nu))|_{\SP_n},A({\mathcal H}_2,(\mu,\nu))|_{\SP_n}) \leq d_n({\mathcal H}_1|_{\SP_n}, {\mathcal H}_2|_{\SP_n}),
\end{equation}
for every \({\mathcal H}_1, {\mathcal H}_2 \in C(\SP,\calP([0,1]))\) and
\( (\mu,\nu) \in \calPP.\)
Indeed, for every $(x,y) \in \SP_n,$
\[
\begin{aligned}
d_W (A({\mathcal H_1}&,(\mu,\nu))(x,y), A({\mathcal H_2},(\mu,\nu))(x,y))\\
&\leq \frac{x}{x+y}\int_0^{\beta} d_W ({\mathcal H_1}(x+k,y), {\mathcal H_2}(x+k,y))
\mu(dk) \\
&\qquad +\frac{y}{x+y}\int_0^{\beta} d_W ({\mathcal H_1}(x,y+k), {\mathcal
H_2}(x,y+k)) \nu(dk)
\\
&\leq \sup_{(x',y')\in \SP_n} d_W ({\mathcal
H_1}(x',y'), {\mathcal H_2}(x',y')).
\end{aligned}
\]
Moreover, if ${\mathcal H} \in C(\SP,\calP([0,1])) $ is Lipschitz on \(\SP_n\)
with Lipschitz constant $K_n,$ then, for every \((\mu_1,\nu_1), (\mu_2,\nu_2) \in \calPP,\)
\begin{equation}\label{eq:opA2r}
d_n(A({\mathcal H},(\mu_1,\nu_1))|_{\SP_n},A({\mathcal H},(\mu_2,\nu_2))|_{\SP_n})
\leq
K_nd_M((\mu_1,\nu_1),(\mu_2,\nu_2)),
\end{equation}
since, for every $(x,y) \in \SP_n,$
\begin{equation}\label{eq:opA2rbis}
\begin{aligned}
d_W (A({\mathcal H}&,(\mu_1,\nu_1))(x,y),A({\mathcal H},(\mu_2,\nu_2))(x,y))\\
&\leq \frac{x}{x+y}\int_0^1 d_W ({\mathcal H}(x+q_{\mu_1}(t),y), {\mathcal
H}(x+q_{\mu_2}(t),y))
dt \\
&\qquad +\frac{y}{x+y}\int_0^1 d_W ({\mathcal H}(x,y+q_{\nu_1}(t)),
{\mathcal H}(x,y+q_{\nu_2}(t)))
dt
\\
&\leq K_n (\tfrac{x}{x+y}d_W(\mu_1,\mu_2)+\tfrac{y}{x+y}d_W(\nu_1,\nu_2)) \\
&\leq K_n d_{M}((\mu_1,\nu_1),(\mu_2,\nu_2)).
\end{aligned}
\end{equation}

Now, for every \(\mathcal H \in C(\SP,\calP([0,1]))\) and \( (\mu,\nu) \in \calPP,\) set \(A^0({\mathcal H},(\mu,\nu))={\mathcal H}\) and, for \(N=1,2,...\) define iteratively
\[
A^N({\mathcal H},(\mu,\nu))=A(A^{N-1}({\mathcal H},(\mu,\nu)),(\mu,\nu)).
\]
Consider \({\mathcal H}_0 \in C(\SP,\calP([0,1]))\) defined by setting $\mathcal H_0 (x,y) = \delta(\frac{x}{x+y})$ for every
\((x,y) \in \SP;\) then
$Z_0^{(\mu,\nu)}(x,y)$  has distribution $\mathcal H_0 (x,y),$
while, for \(N=1,2,...,\) $Z_N^{(\mu,\nu)}(x,y)$ has distribution $A^N(\mathcal H_0,(\mu,\nu))(x,y),$
where, for clarity of exposition, the exponent of the \(Z\) variables is evidence for the reinforcement distributions of the \(RRU\) under consideration.
Note that, for $n\ge 1$, ${\mathcal H_0}$ is a Lipschitz map from
$\SP_n$ to $\calP([0,1])$
with Lipschitz constant $n$. Moreover, it is not difficult to show, with computations analogous to those appearing in \eqref{eq:opA2rbis}, that the operator $A$ preserves the Lipschitz property with the same constant; hence, for $(\mu,\nu) \in \calPP$ and \(N=1,2,...,\) $A^N({\mathcal H_0},(\mu,\nu))$ is a Lipschitz map from $\SP_n$ to $\calP([0,1])$
with Lipschitz constant $n$.

Let \((\mu_1,\nu_1), (\mu_2,\nu_2) \in \calPP,\) \(n,N\geq 1\) and, for ease of notation, set
$${\mathcal H_i}=A^{N-1}({\mathcal H_0},(\mu_i,\nu_i)),$$ for $i=1,2;$ then
\begin{equation}\label{eq:A^N+1}
\begin{aligned}
d_n(A^{N}({\mathcal H_0},(\mu_1,\nu_1))|_{\SP_n},&A^{N}({\mathcal H_0},(\mu_2,\nu_2))|_{\SP_n}) \\
&=
d_n(A({\mathcal H_1},(\mu_1,\nu_1))|_{\SP_n},A({\mathcal H_2},(\mu_2,\nu_2))|_{\SP_n}) \\
& \leq d_n(A({\mathcal H_1},(\mu_1,\nu_1))|_{\SP_n},A({\mathcal H_2},(\mu_1,\nu_1)|_{\SP_n})) \\
& \qquad + d_n(A({\mathcal H_2},(\mu_1,\nu_1))|_{\SP_n},A({\mathcal H_2},(\mu_2,\nu_2))|_{\SP_n}) \\
& \leq d_n({\mathcal H_1}|_{\SP_n},{\mathcal H_2}|_{\SP_n}) + n
d_{M}((\mu_1,\nu_1),(\mu_2,\nu_2)),
\end{aligned}
\end{equation}
the last inequality being a consequence of \eqref{eq:opA1r} and
\eqref{eq:opA2r}.
By iteratively applying \eqref{eq:A^N+1}, it follows that
\[
d_n(A^N({\mathcal H_0},(\mu_1,\nu_1)),A^N({\mathcal H_0},(\mu_2,\nu_2))) \leq
nN d_{M}((\mu_1,\nu_1),(\mu_2,\nu_2)).
\]
Therefore, for every $n\geq 1$ and $\epsilon>0$, if $N=N(\epsilon,n)$ is chosen according to
Lemma~\ref{lem:Z_l,F},
we obtain
\begin{align*}
d_n(\F_{(\mu_1,\nu_1)}|_{\SP_n},\F_{(\mu_2,\nu_2)}|_{\SP_n}) & \leq
 d_n(\F_{(\mu_1,\nu_1)}|_{\SP_n},A^N({\mathcal H_0},(\mu_1,\nu_1))|_{\SP_n}) \\
& \qquad + d_n(A^N({\mathcal H_0},(\mu_1,\nu_1))|_{\SP_n},A^N({\mathcal H_0},(\mu_2,\nu_2))|_{\SP_n})
 \\& \qquad +
d_n(A^N({\mathcal H_0},(\mu_2,\nu_2))|_{\SP_n},\F_{(\mu_2,\nu_2)}|_{\SP_n}) \\
& \leq n N d_{M}((\mu_1,\nu_1),(\mu_2,\nu_2)) \\& \qquad +
\sup_{(x,y)\in \SP_n} \Big[
\bE\big(|Z^{(\mu_1,\nu_1)}_\infty(x,y)-Z^{(\mu_1,\nu_1)}_N(x,y)|\big)
\\& \qquad \qquad +
\bE\big(|Z^{(\mu_2,\nu_2)}_N(x,y)-Z^{(\mu_2,\nu_2)}_\infty(x,y)|\big) \Big]
\\& \leq
n N d_{M}((\mu_1,\nu_1),(\mu_2,\nu_2)) + 2\epsilon.
\end{align*}
This shows that the map  $(\mu,\nu)\mapsto \F_{(\mu,\nu)}|_{\SP_n}$,
from $(\calPP,d_M)$ to $(C(\SP_n,\calP([0,1])),d_n),$ is continuous for every $n.$\\
Hence
the map $(\mu,\nu)\mapsto \F_{(\mu,\nu)}$
from $(\calPP,d_M)$ to $(C(\SP,\calP([0,1])),d)$ is continuous; since $\calPP$ is compact, it is also uniformly continuous.
\end{proof}

%
%
\section{The solution to the functional equation problem} \label{sect:fixed}

 In this section we prove Theorem \ref{theo:FEgeneral} and Theorem \ref{theo:contparameters}. In particular we show existence and uniqueness of the solution to Problem \eqref{eq:system1:F} when the boundary datum  is a generic element of $C([0,1],\calP([0,1])).$ Existence is shown by means of a constructive proof based on the canonical solution described in Section \ref{sect:RRU}. Uniqueness is proved through a fixed point argument.

Given $\varphi\in C([0,1],\calP([0,1])),$ define the map $\Gamma_\varphi:\calP([0,1])
\to\calP([0,1])$ by setting,
for every $\xi \in \calP([0,1]),$
\[
 \Gamma_\varphi(\xi) (B) = \int_0^1 \varphi(t)(B) \xi(dt),
\]
where \(B\) ranges over the Borel sets in \([0,1].\)

\begin{lemma}\label{lem:cont_Gamma}
For any given $\varphi\in C([0,1],\calP([0,1])),$ the map $\Gamma_\varphi$
is
uniformly continuous.
\end{lemma}

\begin{proof}
Since $\varphi \in C([0,1],\calP([0,1]))$, \(\varphi\) is uniformly continuous and bounded:
i.e. for any $\epsilon>0$, there is an $\eta=\eta(\epsilon,\varphi)$ such that
\begin{equation}\label{eq:cont_par_proof1}
d_W(\varphi(t_1),\varphi(t_2))\leq\epsilon, \qquad \text{if }|t_1-t_2|\leq \eta,
\end{equation}
while
\begin{equation}\label{eq:cont_par_proof2}
d_W(\varphi(t_1),\varphi(t_2))\leq 1, \qquad \text{for all }t_1,t_2 \in [0,1].
\end{equation}
Now, take $\xi_1,\xi_2 \in \calP([0,1])$ such that
$d_W(\xi_1,\xi_2)< \epsilon\eta$. We are going to prove that
$d_W(\Gamma_\varphi(\xi_1),\Gamma_\varphi(\xi_2))\leq 2\epsilon$.

Because of 
\eqref{eq:KR}, there is a probability space
$(\widetilde{\Omega},\widetilde{\mathcal{A}},\widetilde{\PP})$ carrying a couple
of random variables
\(X_1,X_2\) such that $X_1\sim \xi_1$, $X_2\sim \xi_2$, and
\(
d_W(\xi_1,\xi_2)=\bE (|X_1-X_2|) \leq \epsilon\eta;
\)
by Markov inequality,
\begin{equation}\label{eq:cont_par_proof3}
\widetilde{\PP} (|X_1-X_2|>\eta) \leq \epsilon.
\end{equation}
On the product probability space $(\widetilde{\Omega}\times [0,1],\widetilde{\mathcal{A}}
\otimes{\mathcal{B}}([0,1]),\widetilde{\PP}\otimes \lambda_{[0,1]})$ define the random variables
\begin{equation*}
\zeta_1(\omega,t)= \inf\Big\{z\colon \int_{[0,z]} \varphi(\eta_1(\omega))(ds) \ge t \Big\}=q_{\varphi(\eta_1(\omega))}(t)
\end{equation*}
and
\begin{equation*}
\zeta_2(\omega,t)= \inf\Big\{z\colon \int_{[0,z]} \varphi(\eta_2(\omega))(ds) \ge t \Big\}=q_{\varphi(\eta_2(\omega))}(t),
\end{equation*}
where \(q_\xi\) indicates the quantile function relative to the probability distribution \(\xi \in \calP([0,1]).\)

For \(i=1,2,\) note that \(\varphi(X_i)\) is the conditional distribution of \(\zeta_i\), given \(X_i;\)
thus, $\zeta_i\sim\Gamma_\varphi(\xi_i).$
Moreover, for all $\omega \in \widetilde{\Omega}$,
\begin{equation*}
\begin{aligned}
d_W(\varphi(X_1(\omega)),\varphi(X_2(\omega)))
&=
\int_0^1 | q_{\varphi(X_1(\omega))}(t)-q_{\varphi(X_2(\omega))}(t)| dt
\\
&= \int_0^1 |\zeta_1(\omega,t)-\zeta_2(\omega,t)|dt.
\end{aligned}
\end{equation*}
Hence,
\begin{multline}\label{eq:cont_par_proof5}
d_W(\Gamma_\varphi(\xi_1),\Gamma_\varphi(\xi_2)) 
\leq \bE(|\zeta_1-\zeta_2|)
\\ 
= \bE \big(\bE\big(|\zeta_1-\zeta_2|\big|X_1,X_2\big) \big)
= \bE \big(d_W(\varphi(X_1),\varphi(X_2))\big).
\end{multline}
Now, let $F=\{|X_1-X_2|>\eta\}$. From \eqref{eq:cont_par_proof1},
\eqref{eq:cont_par_proof2} and
\eqref{eq:cont_par_proof3} one obtains:
\begin{align*}
\bE \big(d_W(\varphi(X_1),\varphi(X_2))\big)&= \bE \big(d_W(\varphi(X_1),\varphi(X_2));F\big)
+ \bE \big(d_W(\varphi(X_1),\varphi(X_2));F^c\big)\\
&\leq \widetilde{\PP}(F)
+ \epsilon\widetilde{\PP}(F^c) \leq 2 \epsilon.
\end{align*}
The last inequality, together with \eqref{eq:cont_par_proof5}, concludes the proof.
\end{proof}

\begin{proof}[Proof of Theorem \ref{theo:FEgeneral}.]

 (i) {\it Existence.} When the boundary datum \(\varphi=\delta,\) the existence of a solution $\G^{\delta}_{(\mu,\nu)}$ is guaranteed by Proposition~\ref{prop:soluz RRU}, and this is ${\mathcal F}_{(\mu,\nu)}$.

Now let $\varphi \in C([0,1],\mathcal{P}([0,1]))$ and define, for all \((x,y) \in \SP,\)
 \begin{equation*}
\G^{\varphi}_{(\mu,\nu)}(x,y) = \Gamma_\varphi(\G^{\delta}_{(\mu,\nu)}(x,y));
\end{equation*}
we are going to show that \(\G^{\varphi}_{(\mu,\nu)}\) is indeed a solution to Problem \eqref{eq:system1:F} when the bondary datum is \(\varphi.\)
In other words, the composition
\[
 (x,y) \mapsto \G^{\delta}_{(\mu,\nu)}(x,y) \mapsto
\Gamma_\varphi(\G^{\delta}_{(\mu,\nu)}(x,y))
\]
gives a solution to Problem~\eqref{eq:system1:F}, i.e.,
\eqref{eq:soluz_gen} holds if the map $\Psi_\varphi$ is defined by setting, for all \(\G \in C(\SP,\calP([0,1])),\)
 \begin{equation*}
\Psi_\varphi(\G)(x,y) = \Gamma_\varphi(\G(x,y))
=\int_0^1 \varphi(t) \,\G(x,y)(dt)
\end{equation*}
with \((x,y)\) ranging over \(\SP.\)

Because of Proposition~\ref{prop:soluz RRU}, $\G^{\delta}_{(\mu,\nu)}$ satisfies
\eqref{eq:system1:func_eq} when the border datum is \(\delta.\) Since
$\Gamma_\varphi$ is linear, this implies that, for all \((x,y) \in \SP,\)
\begin{multline*}
 \Psi_\varphi(\G^{\delta}_{(\mu,\nu)})(x,y) =
\Gamma_\varphi(\G^{\delta}_{(\mu,\nu)}(x,y))
\\
\begin{aligned}
& = \Gamma_\varphi\Big(
\frac{x}{x+y}\int \G^{\delta}_{(\mu,\nu)}(x+k,y)\mu(dk) +
\frac{y}{x+y}\int \G^{\delta}_{(\mu,\nu)}(x,y+k)\nu(dk)
\Big)
\\
& =
\frac{x}{x+y}\int \Psi_\varphi(\G^{\delta}_{(\mu,\nu)})(x+k,y)
\mu(dk) +
\frac{y}{x+y}\int \Psi_\varphi(\G^{\delta}_{(\mu,\nu)})(x,y+k)\nu(dk);
\end{aligned}
\end{multline*}
hence $\Psi_\varphi(\G^{\delta}_{(\mu,\nu)})$
satisfies \eqref{eq:system1:func_eq} when the border datum is \(\varphi.\)
Now, by Lemma~\ref{lem:cont_Gamma},
$\Psi_\varphi(\G^{\delta}_{(\mu,\nu)})$
is a continuous map from $\SP$ to $\calP([0,1])$, being the
composition of the continuous maps $\Gamma_\varphi$ and $\G^{\delta}_{(\mu,\nu)}$; hence \eqref{eq:system1:contin},
\eqref{eq:system1:ax1}, \eqref{eq:system1:ax2}
and \eqref{eq:system1:ax_inf} also hold true.

\medskip

(ii) {\it Uniqueness. Sketch of the argument.} Our argument in \cite[Section~5]{AlettiMaySecchi07} can be easily
extended to this more general situation.

Condition~\eqref{eq:system1:ax_inf} requires $\G$ to be
continuous at the projective infinite points. It is therefore
convenient to transform the space $\SP$ along the
projective automorphism $\tau$ of $\PP^2$ so defined:
\[
(x:y:u) \mathop{\mapsto}\limits^\tau (u:x:x+y).
\]
The automorphism \(\tau\) has the following properties:
\begin{itemize}
\item [-] the space $\SP$ is mapped into the affine space \(\SP^*=[0,\infty)\times [0,1]\);
  \item [-] the positive $x$--axis is mapped into itself by $(x,0)\to (1/x,0)$;
  \item [-] the positive $y$--axis is mapped into the semiline $\{y^*=1, x^*>0\}$ by $(0,y)\mapsto (1/y,1)$;
  \item [-] the projective infinite point relative to the direction $\frac{x}{x+y} = k$ is mapped
  in the point $(0, k)$;
  \item [-] the origin is mapped in the projective infinite point $(1:0:0).$
\end{itemize}
The inverse map of \(\tau\) is $(x^*:y^*:u^*) \mathop{\mapsto}\limits^{\tau^{-1}} (y^*:u^*-y^*:x^*).$
Problem~\eqref{eq:system1:F} can be equivalently formulated on \(\SP^*\) as follows:\\
{\em given} \[ (\mu,\nu) \in \calPP \;\mbox{and the boundary datum}\; \varphi \in C([0,1],\calP([0,1])),\]
{\em find}
\begin{subequations}\label{eq:system1:F*}
\begin{equation}
\label{eq:system1:contin*}
\G^* \in C(\SP^*,\calP ([0,1]))
\end{equation}
such that, for all \((x,y) \in \SP^*\),
\begin{align}\label{eq:system1:func_eq*}
&
\begin{aligned}
\G^*(x^*,y^*) &=
y^* \int \G^*\Big(\frac{x^*}{1+kx^*},\frac {y^* + kx^*}{1+
kx^*}\Big) \mu(dk)+\\
&+(1-y^*) \int \G^*\Big(\frac{x^*}{1+kx^*},
\frac {y^*}{1+ kx^*}\Big) \nu(dk),
\end{aligned}
\\
\label{eq:system1:ax1*}
&\G^*(x^*,0) = \varphi(0),
\\
\label{eq:system1:ax2*}
&\G^*(x^*,1) = \varphi(1),
\\
\label{eq:system1:ax_inf*}
&\G^*(0,y^*) = \varphi(y^*).
\end{align}
\end{subequations}
In fact, \eqref{eq:system1:func_eq*} is just \eqref{eq:system1:func_eq}
in the new coordinates. Indeed, the transformations
\begin{align*}
 \G(x,y)&= \G^*(\tau(x,y))
\\
 \G^*(x^*,y^*)&=
\begin{cases}
 \G(\tau^{-1}(x^*,y^*)) & \text{if }(x^*,y^*)\in(0,\infty)\times
  [0,1];
\\
\lim\limits_{s^*\to 0} \G(\tau^{-1}(s^*,y^*)) & \text{if }x^*=0,y^*\in[0,1],
\end{cases}
\end{align*}
show the equivalence of Problem~\eqref{eq:system1:F} and Problem~\eqref{eq:system1:F*}.


Now, let \(\C_\varphi^*(\SP^*)\) be the space of continuous function \(
  {{\mathcal H}^*} :\SP^*\rightarrow {\calP}([0,1])\) such that, for every \((x^*,y^*)\in \SP^*,\)
\[
{\mathcal H}^*(x^*,0)=\varphi(0),\; {{\mathcal H}^*}(x^*,1)=\varphi(1)\;\mbox{and}\;
{\mathcal H}^*(0,y^*)=\varphi({y^*}).
\]
Define the following operator \(A^*\) mapping \(\C_\varphi^*(\SP^*)\)
into \(\C_\varphi^*(\SP^*)\):
\[
 \begin{aligned}
A^* ({\mathcal H}^*)(x^*,y^*) &=
y^* \int {\mathcal H}^*\Big(\frac{x^*}{1+kx^*},\frac {y^* + kx^*}{1+
kx^*}\Big) \mu(dk)+\\
&+(1-y^*) \int {\mathcal H}^*\Big(\frac{x^*}{1+kx^*},
\frac {y^*}{1+ kx^*}\Big) \nu(dk)
\end{aligned}
\]
with \((x^*,y^*)\) ranging over \(\SP^*.\)

With the same argument used in \cite[Theorem~5.2]{AlettiMaySecchi07},
one can prove that $A^*$ has at most one fixed point; hence Problem~\eqref{eq:system1:F*} has at most
one solution.

\medskip

(iii) {\it Continuity with respect to the boundary datum.} We prove this last part by
showing that
\begin{equation}\label{eq:embedding}
d(\G_{(\mu,\nu)}^{\varphi_1},\G_{(\mu,\nu)}^{\varphi_2})
\leq
d_\infty(\varphi_1,\varphi_2)
\end{equation}
for all $\varphi_1,\varphi_2\in C([0,1], \calP([0,1]))$.
We recall here that (see, e.g., \cite{Gibbs.et.al.02})
\begin{equation}\label{eq:defiWasserEquiv}
d_W(\eta_1,\eta_2) = \sup\Big\{\Big|\int h(t)\eta_1(dt)-
\int h(t) \eta_2(dt)\Big|:
\|h\|_L\leq 1\Big\}
\end{equation}
where $\|h\|_L$ is the Lipschitz norm.
Then, for \(h\) such that $\|h\|_L\leq 1$ and $(x,y)\in\SP$, we get
\begin{multline*}
\Big|\int h(s)\G^{\varphi_1}_{(\mu,\nu)}(x,y)(ds)-
\int h(s)\G^{\varphi_2}_{(\mu,\nu)}(x,y)(ds)\Big|\\
\begin{aligned}
& = \Big|\int h(s)\int\varphi_1(t)(ds) \G^{\delta}_{(\mu,\nu)}(x,y)(dt)
\\
& \qquad \qquad -
\int h(s)\int\varphi_2(t)(ds) \G^{\delta}_{(\mu,\nu)}(x,y)(dt)\Big|\\
& \leq \int \Big| \int h(s) \varphi_1(t)(ds) -\int h(s)\varphi_2(t)(ds) \Big|
\G^{\delta}_{(\mu,\nu)}(x,y)(dt)
\\
& \leq \int d_W(\varphi_1(t),\varphi_2(t))
\G^{\delta}_{(\mu,\nu)}(x,y)(dt) \leq d_\infty(\varphi_1,\varphi_2);
\end{aligned}
\end{multline*}
hence $ d_W(\G^{\varphi_1}_{(\mu,\nu)}(x,y),
\G^{\varphi_2}_{(\mu,\nu)}(x,y))\leq
d_\infty(\varphi_1,\varphi_2)$, again by \eqref{eq:defiWasserEquiv}.
Inequality \eqref{eq:embedding} follows easily.
\end{proof}

\begin{remark}
Given $(\mu,\nu)\in\calPP,$ the inequality \eqref{eq:embedding} can be completed as follows:
for all \(\varphi_1,\varphi_2 \in C([0,1], \calP([0,1]),\)
\begin{equation}\label{eq:embeddingcomplete}
d(\G_{(\mu,\nu)}^{\varphi_1},\G_{(\mu,\nu)}^{\varphi_2})
\leq
d_\infty(\varphi_1,\varphi_2)
\leq
2d(\G_{(\mu,\nu)}^{\varphi_1},\G_{(\mu,\nu)}^{\varphi_2}).
\end{equation}
Hence, for any given $(\mu,\nu)\in\calPP,$ we have an embedding
$$C([0,1], \calP([0,1]))
\mathop{\hookrightarrow}\limits^{\Psi_\varphi} C(\SP, \calP([0,1])).$$
Indeed, for \(n=1,2,...,\) \eqref{eq:system1:ax_inf}
implies that
\[
 d_\infty(\varphi_1,\varphi_2) \leq d_n(\G_{(\mu,\nu)}^{\varphi_1},\G_{(\mu,\nu)}^{\varphi_2});
\]
since $d_n\leq 1$, and thus $d_n\leq 2\frac{d_n}{1+d_n}$, this implies the right inequality in \eqref{eq:embeddingcomplete}.
\end{remark}

\begin{remark}
Let $m$ be the common mean of $(\mu,\nu)\in\calPP$.
For $p\in [m_0/m,1]$,
set
$(\mu^\prime,\nu^\prime)=(p\mu+(1-p)\delta_0,
p\nu+(1-p)\delta_0).$
Then \((\mu^\prime,\nu^\prime) \in \calPP\) and
\(
\G^\varphi_{(\mu^\prime,\nu^\prime)}=\G^\varphi_{(\mu,\nu)}.
\)
\end{remark}

\begin{remark}\label{rem:cont_maps}
Let $h:[0,1]\to[0,1]$ be a continuous function and \(\varphi \in C([0,1],\calP([0,1]))\) a boundary datum.
For \(\xi \in \calP([0,1]),\) denote with \(h \circ \xi\) the distribution of the random variable
\(h(W),\) where \(W\) is a random variable with probability distribution \(\xi.\)
Then
\(
\G^{h \circ \varphi}_{(\mu,\nu)}(x,y)=h\circ \G^{\varphi}_{(\mu,\nu)}(x,y),
\)
for all \((x,y) \in \SP.\)
\end{remark}

\begin{remark}
One may notice that the boundary conditions \eqref{eq:system1:ax1} and \eqref{eq:system1:ax2} are redundant. Indeed,
if \eqref{eq:system1:func_eq} and \eqref{eq:system1:ax_inf} are true for a \(\G: \SP\rightarrow \mathcal{P}([0,1]),\) then \(\G\) satisfies \eqref{eq:system1:ax1} and \eqref{eq:system1:ax2}. To see this, let \(x>0\) and consider \(\G(x,0).\) By iteratively applying \eqref{eq:system1:func_eq}, one obtains
\[
\begin{aligned}
\G(x,0)&=\int_0^\beta \G(x+k_1,0) \mu(dk_1)\\
&=\int_0^\beta\int_0^\beta \G(x+k_1+k_2,0) \mu(dk_1)\mu(dk_2)\\
& \cdots\\
&=\int_0^\beta \cdots\int_0^\beta \G(x+k_1+\cdots+k_n,0)\mu(dk_1)\cdots\mu(dk_n)
\end{aligned}
\]
for all \(n=1,2,...\).
However, because of \eqref{eq:system1:ax_inf}, if $\sum_{i=1}^n k_i \rightarrow \infty$ as $n \rightarrow \infty$, then
\[
\lim_{n \rightarrow \infty} d_W(\G(x+\sum_{i=1}^n k_i,0),\varphi(1))=0.
\]
Hence, the Law of Large Numbers and the Dominated Convergence Theorem imply that
\(\G(x,0) = \varphi(1).\)
The argument for proving that \(\G(0,y)=\varphi(0),\) if \(y>0,\) is analogous.
\end{remark}

\begin{remark}
For \(n=1,2,...\) let \(\gamma_n(x,y)\) be the distribution of \((X_n,Y_n),\) the number of black and white balls respectively present in a \emph{RRU}, with initial composition \((x,y) \in \SP\) and reinforcement distribution \(\mu\) and \(\nu,\) after is has been sampled for the \(n\)-th time.

Since \(\{(X_n,Y_n)\}\) is Markov, if \(\G\) is a solution to Problem~\eqref{eq:system1:F}, it is easy to prove, by induction on \(n,\) that
\[
\G(X,Y) \wedge \gamma_n(x,y)= \G(x,y)
\]
for all \(n\) and \((x,y) \in \SP.\)
However one can show that \eqref{eq:system1:ax_inf} implies that
\[
d_W (\G(X, Y)\wedge \gamma_n(x,y), \varphi\Big(\frac{X}{X + Y}\Big)\wedge \gamma_n(x,y) )
\mathop{\longrightarrow}_{n\to\infty} 0.
\]
Hence
\[
\varphi\Big(\frac{X}{X + Y}\Big)\wedge \gamma_n(x,y)
\mathop{\longrightarrow}^{d_W}_{n\to\infty}
\G(x, y)
\]
and this shows the uniqueness of the solution \(\G.\)
This interesting remark has been pointed out by a referee.
\end{remark}

We are finally in the position to prove Theorem \ref{theo:contparameters}.

\begin{proof}[Proof of Theorem \ref{theo:contparameters}.]
The theorem is true when the boundary datum is $\delta$, as follows from
Proposition~\ref{prop:PtoG} and the fact that $\calPP$ is compact.
For a general boundary datum \(\varphi \in C([0,1],\calP([0,1]))\) the result follows once it is proved that the map \(\Psi_\varphi\) is continuous; but this is a consequence of
Lemma~\ref{lem:cont_Gamma}.
\end{proof}

\section{Diffuse solutions}\label{sect:diffuse}


As an immediate consequence of Proposition \ref{prop:soluz RRU} and \cite[Theorem~3.2]{AlettiMaySecchi08} we have the following result, which is a particular instance of Theorem~\ref{theo:nomassgeneral} and represents a tool for proving it.

\begin{proposition}\label{prop:nomassG}
For all $(x,y)$ in the interior of \(\SP\) and $z\in[0,1]$,
$$\G^{\delta}_{(\mu,\nu)}(x,y)(\{z\})=0.$$
\end{proposition}

\begin{proof}[Proof of Theorem~\ref{theo:nomassgeneral}]
Given \(z\in[0,1],\)
note that $\varphi_z(t)=\varphi(t)(\{z\})$ is a measurable function of
$t$, since it is the
monotone limit of the sequence of continuous functions $k_z^{(n)}$ defined by setting, for \(n=1,2,...\) and \(t \in [0,1],\)
\[
 k_z^{(n)}(t)= \int h_z^{(n)}(s)\varphi(t)(ds),
\]
where
\[
h_z^{(n)}(s)=
\begin{cases}
ns-nz+1, & \text{if }z-1/n\leq s\leq z,
\\
nz-ns+1, & \text{if }z<s\leq z+1/n,
\\
0, & \text{otherwise}.
\end{cases}
\]
The two functions
\[
 \varphi_{z}^-(t) = \sup_{w<z} F_{\varphi(t)}(w),\qquad
 \varphi_{z}^+(t) = F_{\varphi(t)}(z)
\]
are monotonically nonincreasing in $t$, since $\varphi$ is monotone. Moreover,
\(
\varphi_z(t) =\varphi_{z}^+(t)-\varphi_{z}^-(t).
\)
Therefore $\varphi_z$ is a bounded variation function and it thus has
at most a countable number of points of discontinuity. Note that
\[
 \Phi(\{z\}) = \int_0^1 \phi_{z}(t) dt=\int_0^1 (\phi_{z}^+(t)-\phi_{z}^-(t) ) dt.
\]
\medskip

\noindent {\it Proof of part 1.}
Let \((x_0,y_0)\) be a point in the interior of \(\SP\) such that \(\G^\varphi_{(\mu,\nu)}(x_0,y_0)\)
has no point masses in \([0,1].\) By way of contradiction, suppose there is a \(z \in [0,1]\) such that
\(\Phi(\{z\})>0.\)

Since \(\Phi(\{z\})>0,\) there are $\epsilon>0, a>0$ and $z_*\in[0,1]$ such that
\begin{equation}\label{eq:nomass_gen:dim1}
\varphi_z(t)>\epsilon, \qquad\text{for all $t\in I_*=[z_*-a,z_*+a]\cap [0,1]$.}
\end{equation}
Set
\[
\begin{aligned}
R=\Big\{(x,y)\in\SP\colon & x\geq\max(2\beta,\frac{64 \beta}{a^2}),
\\
& y = \big(\frac{1}{\overline{z}}-1\big) x \quad \big(\overline{z}=\frac{x}{x+y}\big),
\\
& \overline{z} \in \big[z_*-\frac{a}{2},z_*+\frac{a}{2}\big]\cap [0,1]
\Big\}.
\end{aligned}
\]
For all $(x,y)\in R,$
Corollary~\ref{cor:scostamento} with $h_0=\frac{a}{4}$ implies that
\[
 \PP\Big(Z_\infty(x,y)\not\in I_*\Big)
 \leq
\PP\Big(|Z_\infty(x,y)- \overline{z}|>\frac{a}{2}\Big)
\leq \frac{4}{a}\sqrt{\frac{\beta}{x}}\leq\frac{1}{2}
\]
and thus $\G^{\delta}_{(\mu,\nu)}(x,y)(I_*)\geq \frac{1}{2}$.
Then, because of \eqref{eq:nomass_gen:dim1}, for all $(x,y)\in R$,
\[
\G^\varphi_{(\mu,\nu)}(x,y)(\{z\}) 
= \int_{0}^1 \varphi_z(t)
\G^{\delta}_{(\mu,\nu)}(x,y)(dt)
\geq \epsilon \int_{I_*} \G^{\delta}_{(\mu,\nu)}(x,y)(dt)
\geq \frac{\epsilon}{2}.
\]
Now, set
$$\tau=\inf\{n \geq 0 \colon(X_n(x_0,y_0),Y_n(x_0,y_0))\in R\}\qquad (\inf\varnothing=\infty);$$
it is not difficult to show that \(\PP(\tau<\infty)=p>0\).
Thus the strong Markov property implies that $\G^{\delta}_{(\mu,\nu)}(x_0,y_0)(I_*)\geq
\frac{p}{2};$ therefore
$\G^\varphi_{(\mu,\nu)}(x_0,y_0)(\{z\})\geq p\frac{\epsilon}{2}$ contradicting the assumption that $\G^\varphi_{(\mu,\nu)}(x_0,y_0)$ has no point masses in \([0,1].\) This concludes the proof of part {\it 1.}
\medskip

\noindent {\it Proof of part 2.}
Let now \(z \in [0,1];\) by assumption $\Phi(\{z\})=0$. Since
 $\varphi_z$
has at most a countable number of points of discontinuity,
  $\varphi_z\geq 0$
 and $\int \varphi_z(t)dt= 0,$
there is a sequence $t_1,t_2,...$ such that
$\varphi_z(t)= 0$ for all $t\in F=(\cup_i\{t_i\})^c$. Then, given any \((x,y)\) in the interior of \(\SP,\)
\[
\begin{aligned}
  \G^\varphi_{(\mu,\nu)}(x,y)(\{z\}) &=
  \int_{[0,1]} \varphi_z(t)\G^{\delta}_{(\mu,\nu)}(x,y)(dt)
\\ &= \int_{[0,1]\cap F} \varphi_z(t)\G^{\delta}_{(\mu,\nu)}(x,y)(dt) +
\sum_{t_i} \varphi_z(t_i)\G^{\delta}_{(\mu,\nu)}(x,y)(\{t_i\})
\\ &= 0
\end{aligned}
\]
the last term being zero because of Proposition \ref{prop:nomassG}.
\end{proof}

\section{Examples}\label{section:examples}

In this section we give explicit descriptions of the solution \(\G^\varphi_{(\mu,\nu)}\)
for some specific choices of the reinforcement distributions \((\mu,\nu) \in \calPP\) and of the boundary datum \(\varphi \in C([0,1],\calP([0,1])).\)
The first example is prototypical since it considers the P\'olya urn scheme and the family of Beta distributions,
whose properties had a central role in originating most of the problems tackled in this paper.

\subsection{The P\'olya urn scheme and the family of Beta distributions}
We indicate with $\text{Beta}(x,y)$ the beta distribution on \([0,1]\) with parameters $(x,y)\in\SP.$
If \((x,y)\) is a point in the interior of \(\SP,\) $\text{Beta}(x,y)$ has a density given by
\[
 f_{\text{Beta}(x,y)}(t)
= \frac{\Gamma(x+y)}{\Gamma(x)\Gamma(y)} t^{x-1}
(1-t)^{y-1},
\]
for \(t\in [0,1];\) it is convenient to indicate with
$\text{Beta}(0,y)$ and $\text{Beta}(x,0)$ the point mass at 0 or at 1, respectively.

The random limit composition of a P\'olya urn with initial composition \((x,y) \in \SP\) and constant reinforcement equal to 1
has distribution $\text{Beta}(x,y);$ indeed a P\'olya urn is a special \emph{RRU} with reinforcements \(\mu=\nu=\delta_1.\) The same result holds when \(\mu=\nu=\text{Bernoulli}(p),\) for $p>0;$ see \cite{DurhamFlournoyLi} and the references therein.
Hence, for \(p>0,\)
\[
\G^{\delta}_{(\delta_1,\delta_1)}(x,y)=\G^{\delta}_{(\text{Bernoulli}(p),\text{Bernoulli}(p))}(x,y)=\text{Beta}(x,y),
\]
for all \((x,y) \in \SP.\)

Different families of distributions related to the Beta can be generated through the P\'olya urn scheme, where both reinforcement distributions are equal to the same point mass, by modifying the boundary datum and solving Problem \eqref{eq:system1:F}.

For instance, let \(\lambda>0,\) define \(\varphi_\lambda:[0,1]\rightarrow \calP([0,1])\) by setting
\(\varphi_\lambda(t)=\delta(t^{1/\lambda}),\) for \(t \in [0,1],\)
and consider Problem \eqref{eq:system1:F} with \(\mu=\nu=\delta_1\) and boundary datum equal to \(\varphi_\lambda.\)
Note that \(\varphi_\lambda\) is monotone and thus the unique solution to the problem is diffuse. Indeed
for \((x,y)\) in the interior \(\SP,\)  the distribution $\G^{\varphi_\lambda}_{(\delta_1,\delta_1)}(x,y)$
has density
\[
f_{\G^{\varphi_\lambda}_{(\delta_1,\delta_1)}(x,y)}(t)= \lambda \frac{\Gamma(x+y)}{\Gamma(x)\Gamma(y)}t^{\lambda x-1}(1-t^\lambda)^{y-1},
\]
for \(t \in [0,1].\) For \(x=1\) and \(y>0,\)  the solution  $\G^{\varphi_\lambda}_{(\delta_1,\delta_1)}(1,y)$ is called
the Kumaraswamy distribution with shape parameters \(\lambda\) and \(y;\) see \cite{Kumaraswamy80}.

A captivating family of distributions is generated along these lines by setting the boundary datum \(\varphi\) to be the function mapping \(t\in[0,1]\) to the exponential distribution with parameter \(t\) truncated to \([0,1]:\) hence, for \(t \in (0,1],\) the density of $\varphi(t)$ is
\[
f_{\varphi(t)} (z) =
\frac{t \exp(-zt)}{1-\exp(-t)} 1_{[0,1]}(z),
\]
while $f_{\varphi(0)} (z) = 1_{[0,1]}(z)$.

Then (\ref{eq:soluz_gen}) and (\ref{eq:soluz_gen:2}) imply that, for \((x,y)\) in the interior \(\SP,\)  the distribution $\G^{\varphi}_{(\delta_1,\delta_1)}(x,y)$
has density
\[
f_{\G^{\varphi}_{(\delta_1,\delta_1)}(x,y)} (z) =
\int_0^1 \frac{t \exp(-zt)}{1-\exp(-t)} 
\frac{\Gamma(x+y)}{\Gamma(x)\Gamma(y)}
t^{x-1} (1-t)^{y-1} 
\,dt.
\]
This density admits an interesting representation in terms of the Hurwitz zeta function.

Indeed, since $(1-\exp(-t))^{-1}=\sum_{n\geq 0} (\exp(-t))^{n},$

\begin{align*}
f_{\G^{\varphi}_{(\delta_1,\delta_1)}(x,y)} (z) &=
\frac{\Gamma(x+y)}{\Gamma(x)\Gamma(y)}
\sum_{n= 0}^{\infty} \int_0^1 t e^{-(z+n)t} t^{x-1} (1-t)^{y-1} \,dt
\\
&=
\frac{\Gamma(x+y)}{\Gamma(x)\Gamma(y)}
\frac{\Gamma(x+1)\Gamma(y)}{\Gamma(x+y+1)}
\sum_{n\geq 0}
\frac{\Gamma(x+y+1)}{\Gamma(x+1)\Gamma(y)}
\int_0^1 e^{-(z+n)t} t^{x} (1-t)^{y-1} \,dt\\
&= \frac{x}{x+y}
\sum_{n\geq 0} M(x+1,x+y+1,-(z+n)) ,
\end{align*}
where $M(a,b,z)$ is the Kummer's (confluent hypergeometric) function.
When the real parts \(\Re(a)\) and \(\Re(b)\) of \(a\) and \(b\) are such that \(\Re(b)>\Re(a)>0,\) the function \(M\) can be represented as a Barnes integral:
\[
M(a,b,-z) = \frac{1}{2\pi i}\frac{\Gamma(b)}{\Gamma(a)}\int_{-i\infty}^{i\infty} \frac{\Gamma(-s)\Gamma(a+s)}{\Gamma(b+s)}z^sds
\]
where the contour of integration separates the poles of $\Gamma(a+s)$,
which are $\{-a,-a+1,-a+2,\ldots\}$, from those of $\Gamma(-s)$,
which are $\{0,1,2,\ldots\}$. Therefore this contour may be taken
in the halfplane $\Re(s)<-1$ whenever $\Re(a)>1$.
Moreover, for $z\in(0,1)$ and $\Re(s)<-1$,
\[
\sum_{n\geq 0} (z+n)^s = \zeta (-s,z)
\]
where $\zeta$ is the Hurwitz zeta function. Therefore
\begin{align*}
f_{\G^{\varphi}_{(\delta_1,\delta_1)}(x,y)} (z) &= \frac{x}{x+y}
\sum_{n\geq 0}
\frac{1}{2\pi i}\frac{\Gamma(x+y+1)}{\Gamma(x+1)}\int_{-i\infty}^{i\infty} \frac{\Gamma(-s)\Gamma(x+1+s)}{\Gamma(x+y+1+s)}(z+n)^s ds \\
&= \frac{x}{x+y}
\frac{1}{2\pi i}\frac{\Gamma(x+y+1)}{\Gamma(x+1)}\int_{-i\infty}^{i\infty} \frac{\Gamma(-s)\Gamma(x+1+s)}{\Gamma(x+y+1+s)}\zeta (-s,z) ds \\
&=
\frac{1}{2\pi i}\frac{\Gamma(x+y)}{\Gamma(x)}\int_{-i\infty}^{i\infty} \frac{\Gamma(-s)\Gamma(x+1+s)}{\Gamma(x+y+1+s)}\zeta (-s,z) ds
\end{align*}
where the contour of integration may be chosen in the halfplane $\Re(s)<-1$, since $\Re(x+1)>1$.

\subsection{Bernoulli reinforcements}
A more intriguing extension which goes well
beyond the P\'olya urn scheme
is obtained by considering reinforcement distributions \((\mu,\nu) \in \calPP\) different from equal point masses; we here treat the case where \(\mu\) and \(\nu\) are scaled Bernoulli distributions with the same mean.
Let
$m\geq k_\mu \geq k_\nu >0,$
and assume that \(\mu\) and \(\nu\) are the distributions of two random variables, say \(R_X\) and \(R_Y,\)
such that \(R_X/k_\mu\) has distribution Bernoulli($m/k_\mu$) while \(R_Y/k_\nu\) has distribution
Bernoulli($m/k_\nu$).

Equation~\eqref{eq:FEgeneralTOTcompact}, with $(\mu,\nu)$ as above, reads
\begin{equation*}
\frac{x}{k_\mu}\Big(\G(x,y)-\G(x+k_\mu,y\big)\Big) +
\frac{y}{k_\nu}\Big(\G(x,y)-\G(x,y+k_\nu\big)\Big)=0,
\end{equation*}
which does not depend on \(m.\)
One easily verifies that the equation is satisfied by the continuous map \(\G:\SP \rightarrow \calP([0,1])\) defined by setting, for all \((x,y) \in \SP,\)
\[
\G(x,y) = \text{Beta}(\frac{x}{k_\mu},\frac{y}{k_\nu}).
\]
%
Moreover, note that,
\[
d_W\big(\text{Beta}(\frac{x}{k_\mu},\frac{y}{k_\nu}),\delta(\frac{{x}{k_\nu}}{{x}{k_\nu} + {y}{k_\mu}})\big)\mathop{\longrightarrow}\limits_{x+y\to\infty} 0.
\]
Hence,
if \(h:[0,1] \rightarrow [0,1]\) is defined by setting
\[
h(t)=\frac{t k_\nu}{t k_\nu + (1-t) k_\mu}
\]
for all \(t \in [0,1],\) then
\[
\G^{h \circ \delta}_{(\mu,\nu)}(x,y)=\text{Beta}(\frac{x}{k_\mu},\frac{y}{k_\nu}),
\]
for \((x,y) \in \SP,\)
is the unique solution to Problem \eqref{eq:system1:F} when \(\mu\) and \(\nu\) are the scaled Bernoulli
distributions defined above and the boundary datum is the continuous map \(h \circ \delta:[0,1] \rightarrow \calP([0,1])\) defined by setting
$$ h \circ \delta(t)= \delta(\frac{t k_\nu}{t k_\nu + (1-t) k_\mu})$$
for all \(t \in [0,1].\)

We now want to find the distribution of the limit composition of a \emph{RRU} whose reinforcements are distributed according to
the scaled Bernoulli distributions $\mu$ and $\nu.$
Note that \(h\) is continuous, monotonically increasing and its inverse is
\[
h^{-1}(u)=\frac{u k_\mu}{u k_\mu + (1-u) k_\nu}
\]
for \(u \in [0,1].\)
Then it follows from Remark \ref{rem:cont_maps} that,
\[
\G^{\delta}_{(\mu,\nu)}(x,y)=\G^{h^{-1}\circ h \circ \delta}_{(\mu,\nu)}(x,y)=h^{-1}\circ  \text{Beta}(\frac{x}{k_\mu},\frac{y}{k_\nu}),
\]
for all \((x,y) \in \SP.\)
For \((x,y)\) in the interior of \(\SP,\) \(\G^\delta_{(\mu,\nu)}(x,y)\) has a density and this is
\begin{equation*}
f_{\G^\delta_{(\mu,\nu)}(x,y)}(t) =
k_\mu^{\frac{y}{k_\nu}} k_\nu^{\frac{x}{k_\mu}}
\frac{\Gamma( x/k_\mu + y/k_\nu)}{
\Gamma( x/k_\mu)\Gamma( y/k_\nu)}
\frac{t^{ \frac{x}{k_\mu}-1}(1-t)^{\frac{y}{k_\nu}-1}
}{[t k_\nu + (1-t) k_\mu]^{ \frac{x}{k_\mu}+ \frac{y}{k_\nu}}}
\end{equation*}
for \(t \in [0,1].\)
Apart from the Polya urn scheme, to the best of our knowledge this is the first example of an \emph{RRU} where
the analytical expression of the density of the urn limit composition is known; notably it has been found by solving Problem~\eqref{eq:system1:F}.

\appendix
\section{Doob decomposition of the \emph{RRU} process}
\numberwithin{equation}{section}
This appendix provides a series of auxiliary results necessary to prove Propositions \ref{prop:soluz RRU} and \ref{prop:PtoG}. We will refer to the notations introduced in Section \ref{sect:RRU}.
For \(n=1,2,...\) let $\mathcal A_n=\sigma(\delta_1,R_X(1),R_Y(1),\ldots,\delta_n,R_X(n),R_Y(n) )$ and consider the filtration \(\{\mathcal A_n\};\)
then, given the initial urn composition \((x,y) \in \SP,\) the Doob's semi-martingale decomposition
of $Z_n(x,y)$ is
\[
Z_n(x,y) = Z_0(x,y) + M_n(x,y) + A_n(x,y)
\]
where $\{M_n\}$ is a zero mean martingale and the previsible process $\{A_n\}$ is eventually increasing (decreasing),
again by \cite[Theorem~2]{Muliere.et.al.06}. We also denote by $\{\langle{M}\rangle_n\}$ the bracket process associated to
$\{M_n\},$ i.e. the previsible process
obtained by the Doob's decomposition of $M^2_n$.

We first provide some auxiliary inequalities.
As a consequence of \cite[Lemma~4.1]{AlettiMaySecchi08}, we can bound the increments
$\Delta A_n$ of the $Z_n$-compensator process
and the increments
$\Delta \langle M\rangle_n$ of the bracket process associated to $\{M_n\}$.
In fact, an easy computation
gives
\begin{equation*}
\Delta {A}_{n+1}
=
\bE(\Delta Z_{n+1}|{{\mathcal A}_n}) = Z_{n}(1-Z_n)
A^*_{n+1}
\end{equation*}
and
\begin{equation*}
\bE((\Delta Z_{n+1})^2|{{\mathcal A}_n})
 =
Z_{n}(1-Z_n) {Z^*_{n+1}}.
\end{equation*}
where
\begin{displaymath}
A^*_{n+1} =\bE\Big( \frac{\frac{R_X(n+1)}{D_n}}{1+\frac{R_X(n+1)}{D_n}} -
\frac{\frac{R_Y(n+1)}{D_n}}{1+\frac{R_Y(n+1)}{D_n}} \Big|{{\mathcal A}_n}
\Big),
\end{displaymath}
and
\begin{displaymath}
Z^*_{n+1} =\bE\Big( (1-Z_n) \Big(\frac{\frac{R_X(n+1)}{D_n}}{1+\frac{R_X(n+1)}{D_n}} \Big)^2
+ Z_n \Big(\frac{\frac{R_Y(n+1)}{D_n}}{1+\frac{R_Y(n+1)}{D_n}} \Big)^2\Big|{{\mathcal A}_n}
\Big).
\end{displaymath}
Now, \cite[Lemma~4.2]{AlettiMaySecchi08} with $m=\int_0^\beta k\mu(dk)=\int_0^\beta k\nu(dk)$ gives
\begin{equation}\label{eq:A^*_n}
|{A}^*_{n+1}|  \leq \frac{m}{m+D_n}
-\frac{m}{\beta+D_n}.
\end{equation}
By applying \cite[Lemma~4.1]{AlettiMaySecchi08} with
$h(x,t) = (\frac{x}{x+t})^2$,
$B_D=[2\beta,\infty)$,
$D=D_n$, $R=R_{X}(n+1)$ or $R=R_{Y}(n+1)$ and ${\mathcal A}={\mathcal A}_n$,
one obtains:
\begin{equation}\label{eq:DeltaZ_n^2}
Z_{n}(1-Z_n) \Big(\frac{m}{m+D_n}\Big)^2 \leq
\bE((\Delta Z_{n+1})^2|{{\mathcal A}_n})
\leq
Z_{n}(1-Z_n) \frac{m\beta}{(\beta+D_n)^2},
\end{equation}
on the set $\{D_n\geq 2\beta\}$. Since
\[
\bE((\Delta Z_{n+1})^2|{{\mathcal A}_n}) =
\bE((\Delta {A}_{n+1}+\Delta {M}_{n+1})^2|{{\mathcal A}_n}) =
(\Delta {A}_{n+1})^2+
\Delta \langle{M}\rangle_{n+1},
\]
if $D_0\geq 2\beta,$ and thus $\beta +D_n\geq 3\beta$,
\eqref{eq:A^*_n} together with
\eqref{eq:DeltaZ_n^2} yields
\begin{equation}\label{eq:diseqDeltaBraket}
\begin{aligned}
\Delta \langle{M}\rangle_{n+1}&
\geq
Z_{n}(1-Z_n) \Big(\frac{m}{m+D_n}\Big)^2 \Big(
1- \Big(\frac{\beta -m}{\beta+D_n}\Big)^2 \Big)
\\
&\geq \frac{8}{9} Z_{n}(1-Z_n) \Big(\frac{m}{m+D_n}\Big)^2,
\\[2mm]
\Delta \langle{M}\rangle_{n+1}&
\leq Z_{n}(1-Z_n) \frac{m\beta}{(\beta+D_n)^2}.
\end{aligned}
\end{equation}
\begin{lemma}
For all $k=1,2,...$,
\begin{equation}\label{eq:1suDN}
\bE(\frac{1}{D_{k}})
\leq
\frac{1+(\beta-m)/{D_0}}{D_0+m(k-1)+\beta}.
\end{equation}
If, in addition, $D_0\geq 2\beta$ then, for all
$k,n=1,2,...,$
\begin{equation}\label{eq:1suc-1sud}
\Big| \bE\Big(\frac{1}{c+D_{k+n}}-\frac{1}{d+D_{k+n}}\Big|{\mathcal
A}_n\Big) \Big|
\leq \frac{\beta - m + d}{m} \Big(\frac{1}{b_k}-\frac{1}{b_{k+1}}\Big)\,,
\end{equation}
when $d \geq c\geq 0$ and $b_k=c+D_n-\beta+mk$.
\end{lemma}
\begin{proof}
Let $\eta^*$ be a random variable independent of ${\mathcal A}_\infty$  and let
$\eta_1$ be a random variable independent of $\sigma({\mathcal A}_\infty,\eta^*)$ and such that
\(\eta_1/\beta\) has distribution Binomial(\(1,m/\beta\)).
Define ${\mathcal A}_{k+n^-}^*=\sigma(\eta^*,{\mathcal A}_{k+n-1},\I(k+n))$; by
\cite[Lemma~4.1]{AlettiMaySecchi08}, if $D>0$ is ${\mathcal
A}_{k+n^-}^*$-measurable
and $0\leq R\leq \beta$ with $\bE(R)=m$ is independent of ${\mathcal A}_{k+n^-}^*$, one obtains
\[
\bE\Big(\frac{1}{D+R}\Big|{\mathcal A}_{k+n^-}^*\Big) \leq
\frac{m}{\beta} \frac{1}{D+\beta} + \frac{\beta - m}{\beta} \frac{1}{D} =
\bE\Big(\frac{1}{D+\eta_1}\Big|{\mathcal A}_{k+n^-}^*\Big) ,
\]
and thus
\begin{multline}\label{eq:induc}
\bE\Big(\frac{1}{D_{k+n}+\eta^*}\Big| {\mathcal A}_{k+n^-}^*\Big) \\
\begin{aligned}
& =
\I(k+n)\bE\Big(\frac{1}{D_{k+n-1}+\eta^*+R_X(k+n)}\Big| {\mathcal
A}_{k+n^-}^*\Big) \\
& \qquad +(1-\I(k+n))\bE\Big(\frac{1}{D_{k+n-1}+\eta^*+R_Y(k+n)}\Big|
{\mathcal A}_{k+n^-}^*\Big)
\\
& \leq \I(k+n)\bE\Big(\frac{1}{D_{k+n-1}+\eta^*+\eta_1}\Big| {\mathcal
A}_{k+n^-}^*\Big) \\
& \qquad +(1-\I(k+n))\bE\Big(\frac{1}{D_{k+n-1}+\eta^*+\eta_1}\Big|
{\mathcal A}_{k+n^-}^*\Big)
\\
& = \bE\Big(\frac{1}{D_{k+n-1}+\eta^*+\eta_1}\Big| {\mathcal A}_{k+n^-}^*\Big).
\end{aligned}
\end{multline}
Therefore, for $c\geq 0$, by applying \eqref{eq:induc} $k$-times, we get
\begin{equation}\label{eq:1suc+Dk+n}
\bE\Big(\frac{1}{D_{k+n}+c}\Big| {\mathcal A}_n\Big) \leq
\bE\Big(\frac{1}{D_{n}+c+\eta_k}\Big|{\mathcal A}_n\Big)
\end{equation}
where $\eta_k$  is independent of
$\sigma({\mathcal A}_\infty)$ and \(\eta_k/\beta\) has distribution Binomial(\(k,m/\beta\)).
Equation \eqref{eq:1suDN} is now a consequence of \cite[Eq.~(21)]{Wooff85}:
if $\tilde{\eta_k}\sim\textrm{Binomial}(k,r)$ and $l>0$,
\begin{equation*}
\bE\Big(\frac{1}{l+\tilde{\eta_k}}\Big) \leq
\Big(1+\frac{1-r}{l}\Big)\frac{1}{l+kr+(1-r)}\,.
\end{equation*}
Apply this to \eqref{eq:1suc+Dk+n} with $n=0$,
$\tilde{\eta_k}=\eta_k/\beta$, $l={D_0}/{\beta}$
and $r={m}/{\beta}$ to obtain \eqref{eq:1suDN}.

Equation \eqref{eq:1suc-1sud} is a consequence of \cite[Eq.~(25)]{Wooff85}:
if $\tilde{\eta_k}\sim\textrm{Binomial}(k,r)$ and $l>1$,
\begin{equation*}
\bE\Big(\frac{1}{l+\tilde{\eta_k}}\Big) \leq \frac{1}{l+kr - (1-r)}\,.
\end{equation*}
Apply this to \eqref{eq:1suc+Dk+n} with
$\tilde{\eta_k}=\eta_k/\beta$, $l={D_n+c}/{\beta}$
(which is greater than $2$) and $r={m}/{\beta}$ to obtain
\[
\bE\Big(\frac{1}{c+D_{n+k}}\Big|{\mathcal A}_n\Big) \leq
\frac{1}{c+D_n+m(k+1)-\beta}.
\]
Jensen's inequality yields
\(\bE ((d+D_{n+k})^{-1} |{\mathcal A}_n ) \geq (d+D_n+mk)^{-1}\),
and thus
\[
\Big| \bE\Big(\frac{1}{c+D_{n+k}}-\frac{1}{d+D_{n+k}}\Big|{\mathcal
A}_n\Big) \Big|
\leq \frac{\beta-m+d-c}{(c+D_n+m(k+1)-\beta)(d+D_n+mk)}.
\]
Since
\[
\frac{1}{b_k}-\frac{1}{b_{k+1}} =
\frac{m}{(c+D_n-\beta+mk)(c+D_n-\beta+m(k+1))}
\]
we get \eqref{eq:1suc-1sud}:
\begin{align*}
  \frac{\Big| \bE\Big(\frac{1}{c+D_{n+k}}-\frac{1}{d+D_{n+k}}\Big|{\mathcal
A}_n\Big) \Big|}{
\frac{1}{b_k}-\frac{1}{b_{k+1}}}
&\leq \frac{\beta-m+d-c}{m} \frac{c+D_n-\beta+mk}{d+D_n+mk}
\\
&\leq \frac{\beta-m+d-c}{m}.
\end{align*}
\end{proof}

The following Lemma~\ref{sub_Martingale} and Lemma~\ref{bracket}
provide inequalities which control the previsible and the martingale
part of the process $Z_n$ respectively;
they require that the initial composition of the urn is
sufficiently large.
\begin{lemma}\label{sub_Martingale}
If $D_0\geq 2\beta$, then
\begin{equation*}
\bE( \sup_r |{A}_{r}| ) \leq \frac{\beta}{D_0}.
\end{equation*}
\end{lemma}
\begin{proof}
Apply \eqref{eq:1suc-1sud} with $n=0$, $c=m$, $d=\beta$.
Equation \eqref{eq:A^*_n} then reads
\[
\bE (|{A}^*_{k+1}| )\leq
(2\beta-m)\Big(
\frac{1}{b_{k}} - \frac{1}{b_{k+1}}
\Big),
\]
if $b_k=km+D_0-(\beta-m)$. Since $A_0=0$,
\begin{equation*}
\begin{aligned}
\bE( \sup_r |{A}_{r}| ) & \leq
\bE \Big( \sum_k
|\Delta {A}_{k+1}| \Big)
\leq
\sum_k \frac{1}{4}\bE (
|{A}^*_{k+1}| )
\\
&\leq
\frac{2\beta-m}{4}\sum_k
\Big(
\frac{1}{b_{k}} - \frac{1}{b_{k+1}}
\Big) = \frac{2\beta-m}{4}\frac{1}{D_0-(\beta-m)}
\\&
\leq
\frac{\beta}{D_0},
\end{aligned}
\end{equation*}
where the last inequality is true because $\beta-m\leq\beta\leq D_0/2$.
\end{proof}

\begin{lemma}\label{bracket}
Let $D_0\geq 2\beta$. For all $n\geq 0$,
\begin{equation*}
\bE(
\langle{M}\rangle_{\infty}-\langle{M}\rangle_{n}
|{{\mathcal A}_n})
\leq \frac{\beta}{D_0}.
\end{equation*}
\end{lemma}

\begin{proof}
Since $Z_{n+k}(1-Z_{n+k})\leq 1/4$, by \eqref{eq:diseqDeltaBraket}, one gets
\[
\Delta \langle{M}\rangle_{n+k+1} \leq
\frac{m\beta}{4(\beta+D_{n+k})^2}
\leq
\frac{m}{4}\Big(\frac{1}{D_{n+k}}
-\frac{1}{\beta+D_{n+k}}\Big)\,.
\]
Apply \eqref{eq:1suc-1sud} with $c=0$ and $d=\beta$, obtaining
\[
\bE(\Delta \langle{M}\rangle_{n+k+1} |{\mathcal A}_n)\leq
\frac{m}{4}\frac{2\beta-m}{m}\Big(
\frac{1}{b_{k}} - \frac{1}{b_{k+1}}
\Big),
\]
if $b_k=km+D_n-\beta$. Thus
\begin{equation*}
\begin{aligned}
\bE(
\langle{M}\rangle_{\infty}-\langle{M}\rangle_{n}
|{{\mathcal A}_n})
&=\bE\Big(
\sum_{k\ge0}
\Delta \langle{M}\rangle_{k+n+1}\Big|{{\mathcal A}_n}\Big)
\\
&\leq \frac{2\beta-m}{4}
\sum_{k\ge0}\Big(
\frac{1}{b_{k}}
- \frac{1}{b_{k+1}}
\Big)
\\&
\leq \frac{2\beta}{2}\frac{1}{2(D_n -\beta)}
\leq \frac{\beta}{D_0},
\end{aligned}
\end{equation*}
since $2(D_n -\beta)\geq 2(D_0 -\beta)\geq D_0$.
\end{proof}

\begin{small}\medskip
\noindent
\textbf{Acknowledgement.}
The authors thank two anonymous referees whose
useful comments helped to clarify some key points
of the paper.
\end{small}


\end{document}